\documentclass[a4paper]{amsart}
\usepackage[leqno]{amsmath}
\usepackage{amssymb}
\usepackage{amscd}
\usepackage{amsthm}
\usepackage{mathrsfs}
\usepackage{mathtools}
\usepackage{bbm}
\usepackage{color}
\usepackage{enumerate}
\usepackage{cite}
\usepackage[utf8]{inputenc}
\usepackage[all,cmtip]{xy}
\usepackage{etoolbox}
\usepackage{tikz}
\usepackage{extarrows}

\newcommand{\Z}{\ensuremath{\mathbb{Z}}}
\newcommand{\Q}{\ensuremath{\mathbb{Q}}}
\newcommand{\F}{\ensuremath{\mathbb{F}}}

\DeclareMathOperator{\Ext}{Ext}

\DeclareMathOperator{\GL}{GL}
\DeclareMathOperator{\PGL}{PGL}
\DeclareMathOperator{\SL}{SL}

\DeclareMathOperator{\Ad}{Ad}
\DeclareMathOperator{\Cores}{Cor}
\DeclareMathOperator{\Res}{Res}
\DeclareMathOperator{\inv}{inv}

\DeclareMathOperator{\sm}{sm}
\DeclareMathOperator{\sph}{sph}
\DeclareMathOperator{\HH}{H}

\newcommand{\into}{\hookrightarrow}
\newcommand{\onto}{\twoheadrightarrow}
\newcommand{\too}{\longrightarrow}								
\newcommand{\mapstoo}{\longmapsto}

\newtheorem{Lem}{Lemma}[section]
\makeatletter
\newlength{\@thlabel@width}%
\newcommand{\thmenumhspace}{\settowidth{\@thlabel@width}{\itshape1.}\sbox{\@labels}{\unhbox\@labels\hspace{\dimexpr-\leftmargin+\labelsep+\@thlabel@width-\itemindent}}}
\makeatother
\newtheorem{Pro}[Lem]{Proposition}
\newtheorem{Thm}[Lem]{Theorem}
\newtheorem{Cor}[Lem]{Corollary}

\theoremstyle{definition}
\newtheorem{Def}[Lem]{Definition}
\newtheorem{Rem}[Lem]{Remark}
\newtheorem{Exa}[Lem]{Example}

\author[L.~Gehrmann]{Lennart Gehrmann}
\address{L.~Gehrmann \\ Fakult\"at f\"ur Mathematik \\ Universit\"at Duisburg-Essen \\ Thea-Leymann-Stra\ss e 9 \\ 45127 Essen \\ Germany}
\email{lennart.gehrmann@uni-due.de}

\title[Gelfand's trick for the spherical derived Hecke algebra]{Gelfand's trick for the spherical derived Hecke algebra}
\begin{document}

\begin{abstract}
Gelfand's trick shows that the spherical Hecke algebra of a $p$-adic split reductive group is commutative.
We adapt this strategy in order to show that the spherical derived Hecke algebra is graded-commutative under mild assumptions on the coefficient ring.
\end{abstract}

\maketitle


\subsection*{Introduction}
In the study of derived structures appearing in the Langlands correspondence (see \cite{Ve2} for an overview of some of the themes) the derived Hecke algebra plays a pivotal role.
Let $G$ be a connected, split reductive group over a $p$-adic field $F$ with residue field $\F_q$ and $l\neq p$ a prime.
In Section 3 of \cite{Ve} Venkatesh proves a derived version of the Satake isomorphism for the spherical derived Hecke algebra $\mathcal{H}(G)^{\sph}_{\Z/l^r\Z}$ of $G$ with coefficients in $\Z/l^r\Z$ under two assumptions:
\begin{itemize}
\item the prime $l$ does not divide the order of the Weyl group of $G$ and
\item $q\equiv 1 \bmod l^{r}$.
\end{itemize}
As a corollary, one deduces that the spherical derived Hecke algebra is graded-commutative under these assumptions.

In this note we adapt Gelfand's trick to the derived setting to prove that the derived Hecke algebra is graded-commutative under less restrictive assumptions:
let $R$ be a commutative ring in which $p$ and the number of $\F_q$-points of the full flag variety associated to $G$ are invertible. Then $\mathcal{H}(G)^{\sph}_{R}$ is graded-commutative.
In particular, we see that independent of the exponent $r$ the derived Hecke algebra $\mathcal{H}(G)^{\sph}_{\Z/l^{r}\Z}$ is graded-commutative for every prime $l$ such that $q\equiv1 \bmod l $ and $l$ does not divide the order of the Weyl group.

Let us remind ourselves of how Gelfand's trick works for the classical spherical Hecke algebra:
one first uses an involution on the group $G$ that sends the subgroup of integral points to itself and induces inversion on a maximal split torus to define an anti-automorphism of the Hecke algebra.
The Cartan decomposition implies that this anti-automorphism is in fact the identity, thus proving that the Hecke algebra is commutative.

The situation in the derived setting is different:
one needs both, the involution and the Cartan decomposition, just to define the (graded) anti-automorphism of the derived Hecke algebra.
The construction of said anti-automorphism works for arbitrary coefficient rings (see Proposition \ref{Gelfand}).
Contrary to what happens in degree $0$ the anti-automorphism might be non-trivial in higher degrees.
But under the above mentioned assumptions we are able to show that it is also an automorphism (see Theorem \ref{theorem}).

In addition, by explicitly computing the spherical derived Hecke algebra in the remaining cases we show that $\mathcal{H}(\GL_2(F))^{\sph}_{\Z/n\Z}$, $\mathcal{H}(\SL_2(F))^{\sph}_{\Z/n\Z}$ and $\mathcal{H}(\PGL_2(F))^{\sph}_{\Z/n\Z}$ are graded-commutative for all integers $n\in\Z$ with $(n,2p)=1$ (see Corollary \ref{cor}).

\subsection*{Acknowledgements}
I am grateful to Preston Wake for pointing out a mistake in an earlier draft of the article.
I thank Sebastian Bartling and Giovanni Rosso for helpful discussions.
I also thank the anonymous referees for their careful reading and useful comments that improved the exposition of this paper.
While working on this manuscript I was visiting McGill University, supported by Deutsche Forschungsgemeinschaft, and I would like to thank these institutions.

\subsection*{Notations}
Throughout this paper all rings are unital but not necessarily commutative.

Given a profinite group $H$ and a commutative ring $R$, we always endow $R$ with the trivial $H$-action and discrete topology and write $\HH^{i}(H,R)$ for the $i$-th continuous cohomology group of $H$ with values in $R$ (see for example \cite{NSW}, Section I.2).
Moreover, we put
$$\HH^{\ast}(H,R)=\bigoplus_{i\geq 0} \HH^{i}(H,R).$$
If $H^{\prime}\subseteq H$ is a closed subgroup, we denote the restriction map on cohomology by
$$\Res_{H^{\prime}}^{H}\colon\HH^{\ast}(H,R)\too\HH^{\ast}(H^{\prime}\hspace{-0.1em},R)$$
If $H^{\prime}\subseteq H$ is an open subgroup, we denote the corestriction map by
$$\Cores_{H^{\prime}}^{H}\colon\HH^{\ast}(H^{\prime}\hspace{-0.1em},R)\too\HH^{\ast}(H,R).$$

\section{Split reductive groups}
We fix a finite extension $F$ of $\Q_p$ with ring of integers $\mathcal{O}_F$, uniformizer $\varpi$ and residue field $\F_q$.
Further, we fix a connected, split reductive group $\underline{G}$ over $F$.
Then $\underline{G}$ is the generic fibre of a group scheme (that we also denote by $\underline{G}$) over $\mathcal{O}_F$, whose special fibre is reductive.
Let $\underline{T}\subseteq\underline{B}\subseteq \underline{G}$ be a maximal split torus and a Borel subgroup, both defined over $\mathcal{O}_F$.
Denote by $\underline{U}\subseteq\underline{B}$ the unipotent radical of $\underline{B}$.
We put $G=\underline{G}(F)$, $T=\underline{T}(F)$, $B=\underline{B}(F)$ and $K=\underline{G}(\mathcal{O}_F).$

We denote the set of characters respectively cocharacters of $\underline{T}$ by $X^\bullet$ respectively $X_\bullet$.
Let $\varphi^{+}\subseteq X^{\bullet}$ be the set of positive roots associated to $\underline{B}$ and set
$$P^+=\left\{ \lambda\in X_\bullet\mid \langle\lambda,\alpha\rangle\geq 0 \mbox{ for all } \alpha \in \varphi^{+} \right\}.$$

\begin{Pro}[Cartan decomposition]
We have:
$$G=KTK.$$
More precisely, the group $G$ is the disjoint union of the double cosets $K\lambda(\varpi)K$ where $\lambda$ runs through the cocharacters in $P^{+}$.

Furthermore, there exists an involution $\sigma$ of $G$ such that $\sigma(K)=K$ and $\sigma(t)=t^{-1}$ for all $t\in T.$
\end{Pro}
\begin{proof}
For the first claim see §4 of \cite{BruhatTits}.
For the second claim see Chapter 8, §2 of \cite{BourLie}.
\end{proof}

\begin{Exa}
In case $\underline{G}=\GL_n$ with $\underline{T}$ the torus of diagonal matrices we may take $\sigma(g)=(g^{t})^{-1}$.
\end{Exa}

The involution $\sigma\colon G\to G,\ g \mapsto g^\sigma$ descends to a map $\sigma\colon G/K \to G/K$.
We let $G$ act diagonally on $G/K \times G/K$, i.e.~$g.(x,y)=(gx,gy).$
The bijection
$$G/K \times G/K \too G/K \times G/K,\quad (x,y)\mapstoo (y^\sigma\hspace{-0.2em},x^\sigma)$$
sends $G$-orbits to $G$-orbits.
\begin{Cor}\label{corollary}
For all $(x,y)\in G/K \times G/K$ we have
$$G.(x,y)=G.(y^{\sigma}\hspace{-0.2em},x^{\sigma}).$$
\end{Cor}
\begin{proof}
By the Cartan decomposition there exist $t\in T$ and $g\in G$ such that $(x,y)=g.(t,1).$
Therefore, we have
\begin{align*}(y^\sigma\hspace{-0.2em},x^\sigma)&=g^{\sigma}.(1,t^{\sigma})\\
&=g^{\sigma}.(1,t^{-1})\\
&=(g^{\sigma}t^{-1}g^{-1}).(x,y).\end{align*}
\end{proof}

\begin{Pro}[Iwasawa decomposition]
We have $G=BK.$
\end{Pro}
\begin{proof}
See §4 of \cite{BruhatTits}.
\end{proof}

\section{Graded-commutative rings and cup products}
A graded ring $A=\bigoplus_{n\geq 0} A_n$ is called graded-commutative if
$$a\ast b= (-1)^{ij} \cdot b\ast a \quad \forall a\in A_i, b\in A_j.$$
A basic example of a graded-commutative ring is the cohomology ring of a group:
let $H$ be a profinite group and $R$ a commutative ring.
We consider the cup product pairing
$$\HH^i(H,R) \times \HH^j(H,R)\too \HH^{i+j}(H,R),\quad (a,b)\mapstoo a\cup b$$
By \cite{NSW}, Proposition 1.4.4, this induces a graded ring structure on the cohomology $\HH^{\ast}(H,R)$. 
Moreover, from the same proposition one immediately deduces the following:
\begin{Pro}\label{group}
The graded ring $\HH^{\ast}(H,R)$ is graded-commutative. 
\end{Pro}

Let us recall a few standard properties of cup products respectively group cohomology in general that we are going to use throughout this article, most of which can be found in Section I.5 of \cite{NSW}.
Suppose $H^{\prime}\subseteq H$ is a closed subgroup.
A straightforward calculation with continuous cochains shows that
\begin{align}\label{rescup}
\Res_{H'}^{H}(a\cup b)= \Res_{H'}^{H}(a) \cup  \Res_{H'}^{H}(b)
\end{align}
for all $a, b \in \HH^\ast(H,R).$
In other words, the restriction map
$$\Res_{H'}^{H}\colon \HH^\ast(H,R)\too \HH^\ast(H',R)$$
is a homomorphism of graded rings.

For a continuous homomorphism $\varphi\colon H_1\to H_2$ between profinite groups we denote by 
$$\varphi^\ast\colon \HH^\ast(H_2,R)\too \HH^\ast(H_1,R)$$ 
its pullback on cohomology groups (see the beginning of \cite{NSW}, Section I.5).
This makes taking continuous cohomology with values in $R$ a contravariant functor from the category of profinite groups to the category of $R$-modules.

If $\iota\colon H'\into H$ is the inclusion of a closed subgroup, taking pullback along $\iota$ is simply the restriction map, i.e.~
$$\iota^\ast=\Res_{H'}^{H}.$$
Identity \eqref{rescup} has the following generalization, which again follows by a straightforward calculation with continuous cochains:
if $\varphi\colon H_1\to H_2$ is a continuous homomorphism between profinite groups, then
\begin{align*}
\varphi^\ast(a\cup b)= \varphi^\ast(a) \cup  \varphi^\ast(b)
\end{align*}
holds for all $a, b \in \HH^\ast(H_2,R).$

An immediate consequence of functoriality for pullbacks is the following:
let $\varphi\colon H_1\to H_2$ be a continuous homomorphism between profinite groups and $H_1'\subseteq H_1$ a closed subgroup.
We put  $H_2'=\varphi(H_1')$.
Then the diagram
\begin{center}
\begin{tikzpicture}
    \path 	(0,0) 	node[name=A]{$\HH^{\ast}(H_2,R)$}
		(4,0) 	node[name=B]{$\HH^{\ast}(H_1,R)$}
		(0,-1.5) 	node[name=C]{$\HH^{\ast}(H_2',R)$}
		(4,-1.5) 	node[name=D]{$\HH^{\ast}(H_1',R)$};
    \draw[->] (A) -- (C) node[midway, left]{$\Res^{H_2}_{H_2'}$};
    \draw[->] (A) -- (B) node[midway, above]{$\varphi^\ast$};
    \draw[->] (C) -- (D) node[midway, above]{$\varphi^\ast$};
    \draw[->] (B) -- (D) node[midway, right]{$\Res^{H_1}_{H_1'}$};
\end{tikzpicture} 
\end{center}
is commutative.

If moreover $\varphi$ is an isomorphism and $H_1'\subseteq H_1$ is an open subgroup, the following diagram is commutative:
\begin{center}
\begin{tikzpicture}
    \path 	(0,0) 	node[name=A]{$\HH^{\ast}(H_2',R)$}
		(4,0) 	node[name=B]{$\HH^{\ast}(H_1',R)$}
		(0,-1.5) 	node[name=C]{$\HH^{\ast}(H_2,R)$}
		(4,-1.5) 	node[name=D]{$\HH^{\ast}(H_1,R)$};
    \draw[->] (A) -- (C) node[midway, left]{$\Cores^{H_2}_{H_2'}$};
    \draw[->] (A) -- (B) node[midway, above]{$\varphi^\ast$};
    \draw[->] (C) -- (D) node[midway, above]{$\varphi^\ast$};
    \draw[->] (B) -- (D) node[midway, right]{$\Cores^{H_1}_{H_1'}$};
\end{tikzpicture} 
\end{center}

Besides restrictions we will frequently use the following special case of pullbacks:
for a profinite group $H$, a closed subgroup $H'\subseteq H$ and an element $g\in H$ we denote by
$$\Ad(g)\colon H'\to g H'\hspace{-0.1em}g^{-1},\quad h\mapstoo ghg^{-1}$$
the conjugation isomorphism.
In accordance with the notation used in \cite{Ve} we put
$$[g]^\ast=\Ad(g)^\ast\colon \HH^\ast(g H'\hspace{-0.1em}g^{-1}, R) \too \HH^\ast(H', R).$$
To avoid confusion let us remark that the maps $[g]^\ast$ are called conjugation in \cite{NSW} and denoted by $(g^{-1})_\ast$.
An explicit calculation with continuous cochains (see \cite{KB}, Section III.8 for the non-continuous case) shows that, if $g$ is an element of $H'$, the map
$$[g]^\ast\colon \HH^\ast(H', R) \too \HH^\ast(H', R)$$
is the identity.

\section{Derived Hecke algebra}
We fix a commutative ring $R$ in which $p$ is invertible.
An $R[G]$-module $M$ is called smooth if the stabilizer of each element $m\in M$ is open in $G$.
The category $\mathcal{C}^{\sm}_R(G)$ of smooth $R[G]$-modules is an abelian category and has enough projectives (see Appendix A.2 of \cite{Ve}).

\begin{Def}
The \emph{spherical derived Hecke algebra} of $G$ with coefficients in $R$ is the graded algebra
$$\mathcal{H}(G)^{\sph}_R=\mathcal{H}(G,K)_R=\Ext^{\ast}_{\mathcal{C}^{\sm}_R(G)}(R[G/K],R[G/K]).$$
\end{Def}

We will use a more explicit description of $\mathcal{H}(G)^{\sph}_R$ (cf.~\cite{Ve}, Section 2.3).
An isomorphism between the spherical derived Hecke algebra and the algebra described below is constructed in Appendix A of \cite{Ve}.
The following notations will come in handy:
given a tuple $(x_1\ldots,x_n)\in (G/K)^n$ we write $G_{x_1,\ldots,x_n}$ for its stabilizer in $G$.
By definition we have $G_{x_1,\ldots,x_n}=\bigcap_{i=1}^{n}G_{x_i}.$

An element $h\in \mathcal{H}(G)^{\sph}_R$ is a collection of elements $h(x,y)\in \HH^{\ast}(G_{x,y},R)$ for all $(x,y) \in G/K \times G/K$ such that
\begin{itemize}
\item $h$ is $G$-invariant, i.e.~$[g]^{\ast}h(gx,gy)=h(x,y)$ for all $g\in G$ and
\item $h$ has finite support modulo $G$, i.e.~$h(x,y)=0$ outside a finite set of $G$-orbits in $G/K\times G/K.$
\end{itemize}
The addition of two elements $h_1$ and $h_2$ is pointwise.
The degree $n$-part of $\mathcal{H}(G)^{\sph}_R$ consists of all elements which take values in $\HH^{n}(\cdot,R).$
Let $x$ and $z$ be elements of $G/K$.
The quantity
$$\Cores^{G_{x,z}}_{G_{x,y,z}}(\Res^{G_{x,y}}_{G_{x,y,z}} (h_1(x,y))\cup \Res^{G_{y,z}}_{G_{x,y,z}} (h_2(y,z)))$$
only depends on the $G_{x,z}$-orbit of $y\in G/K$.
We have the following formula for the product in $\mathcal{H}(G)^{\sph}_R$:
\[\label{product}\tag{\mbox{$\dagger$}}
(h_1 \ast h_2)(x,z)=\hspace{-0.3em}\sum_{y\in G_{x,z}\backslash G/K}\hspace{-0.6em} \Cores^{G_{x,z}}_{G_{x,y,z}}(\Res^{G_{x,y}}_{G_{x,y,z}} (h_1(x,y))\cup \Res^{G_{y,z}}_{G_{x,y,z}} (h_2(y,z))).
\]

\section{Gelfand's trick}
As before, $R$ denotes a commutative ring in which $p$ is invertible.
For any two elements $x$ and $y$ of $G/K$ there exists an element $g_{x,y}\in G$ such that $g_{x,y}(x,y)=(y^{\sigma}\hspace{-0.2em},x^{\sigma})$ by Corollary \ref{corollary}.
Clearly, $g_{x,y}$ is not unique, but its image in $G/G_{x,y}$ is.

It is easy to see that $g_{x,y}G_{x,y}g_{x,y}^{-1}=G_{x^{\sigma}\hspace{-0.2em},y^{\sigma}}$ and $\sigma(G_{x,y})=G_{x^{\sigma}\hspace{-0.2em},y^{\sigma}}$.
Thus we get a homomorphism
$$[g_{x,y}]^{\ast}\sigma^{\ast}\colon \HH^{\ast}(G_{x,y},R)\too\HH^{\ast}(G_{x,y},R).$$
It is independent of the choice of $g_{x,y}$.
Indeed, any other choice is of the form $g_{x,y}'=g_{x,y}g$ with $g\in G_{x,y}.$
We have $[g_{x,y}']^{\ast}\sigma^{\ast}=[g]^{\ast}[g_{x,y}]^{\ast}\sigma^{\ast}$ and, thus, independence follows from the fact that $[g]^\ast$ acts as the identity on $\HH^{\ast}(G_{x,y},R)$.

Given an element $h\in\mathcal{H}(G)^{\sph}_R$ of the spherical derived Hecke algebra we define the collection of elements $h^{\sigma}$ by
$$h^{\sigma}(x,y)=[g_{x,y}]^{\ast}\sigma^{\ast} h(x,y).$$
This collection defines an element of the derived Hecke algebra:
let $g$ be an element of $G$.
From the equality $\Ad(g)\circ \sigma = \sigma \circ \Ad(g^\sigma)$ we deduce that
\begin{align}\label{adjoint}
\sigma^{\ast}[g]^\ast=[g^\sigma]^{\ast}\sigma^{\ast}.
\end{align}
Therefore, we get that
\begin{align*}
h^{\sigma}(x,y)
&= [g_{x,y}]^{\ast}\sigma^{\ast} h(x,y)\\
&= [g_{x,y}]^{\ast}\sigma^{\ast} [g^\ast]h(gx,gy)\\
&= [g_{x,y}]^{\ast}[g^\sigma]^\ast \sigma^{\ast} h(gx,gy).
\end{align*}
By definition we have $g^\sigma g_{x,y}g^{-1} .(gx,gy)= ((gy)^\sigma\hspace{-0.2em}, (gx)^{\sigma}).$
Thus, we may choose $g_{gx,gy}=g^\sigma g_{x,y}g^{-1}$ and see that
\begin{align*}
[g_{x,y}]^{\ast}[g^\sigma]^\ast \sigma^{\ast} h(gx,gy)
&= [g]^\ast[g_{gx,gy}]^{\ast} \sigma^{\ast} h(gx,gy)\\
&= [g]^\ast h^{\sigma}(gx,gy),
\end{align*}
that is, the collection $h^\sigma$ is $G$-invariant.
That it has finite support modulo $G$ follows directly from its construction.

\begin{Lem}
The $R$-linear map
$$\mathcal{H}(G)^{\sph}_R\too \mathcal{H}(G)^{\sph}_R,\quad h\mapstoo h^\sigma$$
is an involution, i.e.~the equality $(h^\sigma)^\sigma=h$ holds for all $h\in \mathcal{H}(G)^{\sph}_R.$
\end{Lem}
\begin{proof}
Since $g_{x,y}^\sigma .(x^\sigma\hspace{-0.2em},y^\sigma)=((g_{x,y}x)^\sigma\hspace{-0.2em},(g_{x,y}y)^\sigma)=(y,x)$ we see that we may take $g_{x,y}^\sigma$ as a choice for $g_{x^{\sigma}\hspace{-0.2em},y^{\sigma}}.$
Thus, by \eqref{adjoint} we deduce the equality
\begin{align*}
[g_{x,y}]^\ast\sigma^\ast[g_{x,y}]^\ast\sigma^\ast
=[g_{x,y}]^\ast[g_{x,y}^{\sigma}]^\ast\sigma^\ast\sigma^\ast
=[g_{x^{\sigma}\hspace{-0.2em},y^{\sigma}}g_{x,y}]^\ast
\end{align*}
of maps from $H^\ast(G_{x,y},R)$ to itself.
By definition the product $g_{x^{\sigma}\hspace{-0.2em},y^{\sigma}}g_{x,y}$ stabilizes the element $(x,y)$, i.e.~$g_{x^{\sigma}\hspace{-0.2em},y^{\sigma}}g_{x,y}\in G_{x,y}.$
Hence, $[g_{x^{\sigma}\hspace{-0.2em},y^{\sigma}}g_{x,y}]^\ast$ acts as the identity on $H^\ast(G_{x,y},R)$ and the claim follows.
\end{proof}
 
\begin{Pro}\label{Gelfand}
Let $h_1$ and $h_2$ elements of $\mathcal{H}(G)^{\sph}_R$ of degree $i$ respectively $j$.
The following equality holds:
$$(h_1\ast h_2)^{\sigma}= (-1)^{ij} \cdot h_2^{\sigma}\ast  h_1^{\sigma}$$
\end{Pro}
\begin{proof}
Let $x,z \in G/K$ be arbitrary elements.
We have:
\begin{align*}
&(h_1\ast h_2)(x,z)\\=&[g_{x,z}]^{\ast}(h_1\ast h_2)(z^{\sigma}\hspace{-0.2em},x^{\sigma})\\
=&[g_{x,z}]^{\ast}\hspace{-0.6em}\sum_{y\in G_{z^{\sigma}\hspace{-0.2em},x^{\sigma}}\backslash G/K}\hspace{-1.2em}
\Cores^{G_{z^{\sigma}\hspace{-0.2em},x^{\sigma}}}_{G_{z^{\sigma}\hspace{-0.2em},y,x^{\sigma}}}(\Res^{G_{z^{\sigma}\hspace{-0.2em},y}}_{G_{z^{\sigma}\hspace{-0.2em},y,x^{\sigma}}} (h_1(z^{\sigma}\hspace{-0.2em},y))\cup \Res^{G_{y,x^{\sigma}}}_{G_{z^{\sigma}\hspace{-0.2em},y,x^{\sigma}}} (h_2(y,x^{\sigma})))\\
=&(-1)^{ij}[g_{x,z}]^{\ast}\hspace{-0.6em}\sum_{y\in G_{z^{\sigma}\hspace{-0.2em},x^{\sigma}}\backslash G/K}\hspace{-1.2em}\Cores^{G_{z^{\sigma}\hspace{-0.2em},x^{\sigma}}}_{G_{z^{\sigma}\hspace{-0.2em},y,x^{\sigma}}}(\Res^{G_{y,x^{\sigma}}}_{G_{z^{\sigma}\hspace{-0.2em},y,x^{\sigma}}} (h_2(y,x^{\sigma}))\cup \Res^{G_{z^{\sigma}\hspace{-0.2em},y}}_{G_{z^{\sigma}\hspace{-0.2em},y,x^{\sigma}}} (h_1(z^{\sigma}\hspace{-0.2em},y))).
\end{align*}
The first equality follows from the definition of the derived Hecke algebra and the last equality follows from Proposition \ref{group}.
Since $\sigma(G_{x,z})=G_{x^{\sigma}\hspace{-0.2em},z^{\sigma}}$ we deduce that $\sigma$ sends $G_{x,z}$-orbits to $G_{x^{\sigma}\hspace{-0.2em},z^{\sigma}}$-orbits in $G/K$.
Therefore the following equality holds:
\begin{align*}
&\sum_{y\in G_{z^{\sigma}\hspace{-0.2em},x^{\sigma}}\backslash G/K}\hspace{-1.2em}\Cores^{G_{z^{\sigma}\hspace{-0.2em},x^{\sigma}}}_{G_{z^{\sigma}\hspace{-0.2em},y,x^{\sigma}}}(\Res^{G_{y,x^{\sigma}}}_{G_{z^{\sigma}\hspace{-0.2em},y,x^{\sigma}}} (h_2(y,x^{\sigma}))\cup \Res^{G_{z^{\sigma}\hspace{-0.2em},y}}_{G_{z^{\sigma}\hspace{-0.2em},y,x^{\sigma}}} (h_1(z^{\sigma}\hspace{-0.2em},y)))\\
=&\sum_{y\in G_{x,z}\backslash G/K}\hspace{-1.2em}\Cores^{G_{z^{\sigma}\hspace{-0.2em},x,^{\sigma}}}_{G_{z^{\sigma}\hspace{-0.2em},y^{\sigma}\hspace{-0.2em},x^{\sigma}}}(\Res^{G_{y^{\sigma}\hspace{-0.2em},x^{\sigma}}}_{G_{z^{\sigma}\hspace{-0.2em},y^{\sigma}\hspace{-0.2em},x^{\sigma}}} (h_2(y^{\sigma}\hspace{-0.2em},x^{\sigma}))\cup \Res^{G_{z^{\sigma}\hspace{-0.2em},y^{\sigma}}}_{G_{z^{\sigma}\hspace{-0.2em},y^{\sigma}\hspace{-0.2em},x^{\sigma}}} (h_1(z^{\sigma}\hspace{-0.2em},y^{\sigma})))\\
=&\sigma^{\ast}\hspace{-0.2em}\sum_{y\in G_{x,z}\backslash G/K}\hspace{-1.2em}\Cores^{G_{z,x}}_{G_{z,y,x}}\sigma^{\ast}(\Res^{G_{y^{\sigma}\hspace{-0.2em},x^{\sigma}}}_{G_{z^{\sigma}\hspace{-0.2em},y^{\sigma}\hspace{-0.2em},x^{\sigma}}} (h_2(y^{\sigma}\hspace{-0.2em},x^{\sigma}))\cup \Res^{G_{z^{\sigma}\hspace{-0.2em},y^{\sigma}}}_{G_{z^{\sigma}\hspace{-0.2em},y^{\sigma}\hspace{-0.2em},x^{\sigma}}} (h_1(z^{\sigma}\hspace{-0.2em},y^{\sigma})))\\
=&\sigma^{\ast}\hspace{-0.2em}\sum_{y\in G_{x,z}\backslash G/K}\hspace{-1.2em}\Cores^{G_{z,x}}_{G_{z,y,x}}(\Res^{G_{y,x}}_{G_{z,y,x}}\sigma^{\ast}([g_{x^\sigma\hspace{-0.2em}, y^{\sigma}}]^{\ast} h_2(x,y))\cup \Res^{G_{z,y}}_{G_{z,y,x}} (\sigma^{\ast}[g_{y^{\sigma}\hspace{-0.2em},z^{\sigma}}]^{\ast}h_1(y,z)))
\end{align*}
The claim follows from the identity $\sigma^{\ast}[g_{x^{\sigma}\hspace{-0.2em},y^{\sigma}}]^{\ast}=[g_{x,y}]^{\ast}\sigma^{\ast}$.
\end{proof}
If $h\in \mathcal{H}(G)^{\sph}_R$ has degree $0$, then $h^{\sigma}=h$ by definition.
Hence, the degree $0$ part of the spherical derived Hecke algebra, which is just the usual spherical Hecke algebra, is commutative.
This is the classical trick by Gelfand.

As in Venkatesh's proof of the derived Satake isomorphism we want to restrict to the cohomology of the torus $\underline{T}(\mathcal{O}_F).$
The next lemma gives a criterion for when the restriction map is injective.
It is a slight variant of \cite{Ve}, Lemma 3.7.

\begin{Lem}\label{reduction}
Assume that $|\underline{G}(\F_q)/\underline{B}(\F_q)|$ is invertible in $R$.
Let $\mathcal{G}$ be a subgroup of $\underline{G}(\F_q)$ that contains $\underline{T}(\F_q)$.
Then the restriction map
$$\Res_{\underline{T}(\F_q)}^{\mathcal{G}}\colon\HH^{\ast}(\mathcal{G},R)\too \HH^{\ast}(\underline{T}(\F_q),R)$$
is injective.
\end{Lem}
\begin{proof}
The composition $\Cores_{\underline{B}(\F_q)}^{\mathcal{G}}\circ\Res_{\underline{B}(\F_q)}^{\mathcal{G}}$ equals multiplication with $|\mathcal{G}/\underline{T}(\F_q)|$.
This is a divisor of
$$|\underline{G}(\F_q)/\underline{T}(\F_q)|=|\underline{G}(\F_q)/\underline{B}(\F_q)|\cdot |\underline{B}(\F_q)/\underline{T}(\F_q)|.$$
Since $|\underline{B}(\F_q)/\underline{T}(\F_q)|$ is a power of $p$, the assertion follows from our assumptions.
\end{proof}

\begin{Lem}\label{cocycle}
Let $A\subseteq \underline{T}(\mathcal{O}_F)$ be a finite group of order prime to $p$.
It acts on $\underline{U}(\mathcal{O}_F)$ via conjugation.
The non-abelian cohomology group $\HH^{1}(A,\underline{U}(\mathcal{O}_F))$ is trivial.
\end{Lem}
\begin{proof}
The group $\underline{U}(\mathcal{O}_F)$ has an $A$-stable filtration, whose graded pieces are commutative pro-$p$ groups.
But the higher cohomology of such groups vanishes by assumption.
Thus, the claim follows by induction.
\end{proof}

\begin{Thm}\label{theorem}
Assume that $\left|\underline{G}(\F_q)/\underline{B}(\F_q)\right|$ is invertible in $R$.
Then the spherical derived Hecke algebra $\mathcal{H}(G)^{\sph}_R$ is graded-commutative.
\end{Thm}

\begin{proof}
The theorem follows from Proposition \ref{Gelfand} and the following claim:
we have
$$(h_1\ast h_2)^{\sigma}=  h_1^{\sigma}\ast  h_2^{\sigma}.$$
By the Cartan decomposition it is enough to check the identity
\begin{align}\label{toshow}
(h_1\ast h_2)^{\sigma}(t,1)=  (h_1^{\sigma}\ast  h_2^{\sigma})(t,1)
\end{align}
in $\HH^{\ast}(G_{t,1},R)$ for all $t\in T$.
Put $A=G_{t,1}\cap T=\underline{T}(\mathcal{O}_F).$
The kernel of the reduction map $A\onto \underline{T}(\F_q)$ is a pro-p group.
Hence, by Lemma \ref{reduction} and our running assumption that $p$ is invertible in $R$ the restriction map
$$\Res_{A}^{G_{t,1}}\colon \HH^{\ast}(G_{t,1},R)\too \HH^{\ast}(A,R)$$
is injective and, thus, it is enough to prove \eqref{toshow} after taking restriction to the cohomology of $A$.

We have $t^{-1}.(t,1)=(1,t^{-1})=(1,t^\sigma)$ and thus, we may choose $g_{t,1}=t^{-1}.$
Every element of $A$ commutes with $t^{-1}$ and, by definition, the involution $\sigma$ induces the inversion map
$$\inv\colon A\too A,\quad a\mapstoo a^{-1}$$ on $A$.
Therefore, the diagram
\begin{center}
\begin{tikzpicture}
    \path 	(0,0) 	node[name=A]{$\HH^{\ast}(G_{t,1},R)$}
		(6,0) 	node[name=B]{$\HH^{\ast}(G_{t,1},R)$}
		(0,-1.5) 	node[name=C]{$\HH^{\ast}(A,R)$}
		(6,-1.5) 	node[name=D]{$\HH^{\ast}(A,R)$};
    \draw[->] (A) -- (C) node[midway, left]{$\Res_A^{G_{t,1}}$};
    \draw[->] (A) -- (B) node[midway, above]{$[g_{t,1}]^{\ast}\sigma^{\ast}$};
    \draw[->] (C) -- (D) node[midway, above]{$\inv^{\ast}$};
    \draw[->] (B) -- (D) node[midway, right]{$\Res_A^{G_{t,1}}$};
\end{tikzpicture} 
\end{center}
is commutative.

Plugging in the explicit formula \eqref{product} for the product we get
\begin{align*}
&\Res_A^{G_{t,1}}((h_1\ast h_2)^{\sigma}(t,1))
=\inv^{\ast} (\Res_A^{G_{t,1}} ((h_1\ast h_2)(t,1)))\\
=&\sum_{y\in G_{t1}\backslash G/K} \inv^{\ast} (\Res_A^{G_{t,1}} (\Cores^{G_{t,1}}_{G_{t,y,1}}(\Res^{G_{t,y}}_{G_{t,y,1}} (h_1(t,y))\cup \Res^{G_{y,1}}_{G_{t,y,1}} (h_2(y,1))))).
\end{align*}
By Proposition 1.5.6 of \cite{NSW} (compare also with equation (46) of \cite{Ve}) we may write
\begin{align}\label{multres}
\begin{split}
&\inv^{\ast} (\Res_A^{G_{t,1}} (\Cores^{G_{t,1}}_{G_{t,y,1}}(\Res^{G_{t,y}}_{G_{t,y,1}} (h_1(t,y))\cup \Res^{G_{y,1}}_{G_{t,y,1}} (h_2(y,1)))))\\
=&\sum_{y^{\prime}} \inv^{\ast} (\Cores^{A}_{A_{y^{\prime}}}(\Res^{G_{t,y^{\prime}}}_{A_{y^\prime}} (h_1(t,y^{\prime}))\cup \Res^{G_{y^{\prime}\hspace{-0.1em},1}}_{A_{y^\prime}} (h_2(y^{\prime}\hspace{-0.1em},1))))\\
=&\sum_{y^{\prime}} \Cores^{A}_{A_{y^{\prime}}}(\inv^{\ast} (\Res^{G_{t,y^{\prime}}}_{A_{y^\prime}} (h_1(t,y^{\prime}))\cup \Res^{G_{y^{\prime}\hspace{-0.1em},1}}_{A_{y^\prime}} (h_2(y^{\prime}\hspace{-0.1em},1)))),
\end{split}
\end{align} 
where $y^{\prime}$ runs through a set of representatives of $A$-orbits in $G_{t1}y\subseteq G/K$ and $A_{y^\prime}$ denotes the stabilizer of $y^{\prime}$ in $A$.
By a similar computation we see that
$$\Res_A^{G_{t,1}}((h_1^\sigma \ast h_2^\sigma )(t,1))=  \sum_{y}\sum_{y^{\prime}} \Cores^{A}_{A_{y^{\prime}}}( (\Res^{G_{t,y^{\prime}}}_{A_{y^\prime}} (h_1^{\sigma}(t,y^{\prime}))\cup \Res^{G_{y^{\prime}\hspace{-0.1em},1}}_{A_{y^\prime}} (h_2^{\sigma}(y^{\prime}\hspace{-0.1em},1)))),$$
where $y$ and $y'$ run through the same sets as above.

Thus, it is enough to prove the following identities for all $x\in G/K$, $t\in T$ and $h\in\mathcal{H}(G)^{\sph}_R$:
\begin{align*}
\inv^{\ast}(\Res^{G_{t,x}}_{A_{x}} (h(t,x)))&=\Res^{G_{t,x}}_{A_{x}}(h^{\sigma}(t,x))\\
\intertext{and}
\inv^{\ast}(\Res^{G_{t,x}}_{A_{x}} (h(x,t)))&=\Res^{G_{t,x}}_{A_{x}}(h^{\sigma}(x,t)).
\end{align*}

We only prove the first equality, the proof of the second being identical.
By taking pullback along $\Ad(t)$ we may reduce to the case $t=1$.
Let $l$ be a prime dividing the order of $T(\F_q)$ but not dividing $|\underline{G}(\F_q)/\underline{B}(\F_q)|$.
We write $A_{x,l}\subseteq A_x$ for the $l$-Sylow subgroup of $A_x$.
It is enough to show that
\begin{align}\label{lastclaim}
\inv^{\ast}(\Res^{G_{1,x}}_{A_{x,l}} (h(1,x)))=\Res^{G_{1,x}}_{A_{x,l}}(h^{\sigma}(1,x))
\end{align}
holds for each such $l$ because $A_x$ is the product of the groups $A_{x,l}$ and a profinite group, whose order is invertible in $R$.

In order to prove \eqref{lastclaim} it is enough to show that we can find a matrix $g_{1,x}$ for the pair $(1,x)$ that lies in the centralizer of $A_{x,l}.$
This is the content of the next two lemmas.
\end{proof}

\begin{Lem}\label{Step1}
There exists a representative $g\in G$ of $x$ that lies in the centralizer of $A_{x,l}$.
\end{Lem}
\begin{proof}
By the Iwasawa decomposition there exists an element $b\in B$ such that $bK=x$.
We define $u_a\in \underline{U}(F)$ by the identity $aba^{-1}=bu_{a}$.
Since $A_{x,l}$ fixes $x$ and $A_{x,l}$ is a subgroup of $K$, we see that $u_a$ is in fact an element of $\underline{U}(\mathcal{O}_F)=\underline{U}(F)\cap K.$
The assignment $a\mapsto u_a$ defines a $1$-cocycle of $A_{x,l}$ with values in $\underline{U}(\mathcal{O}_F)$.
Lemma \ref{cocycle} implies that there exists an element $u\in \underline{U}(\mathcal{O}_F)$ such that $u_a=u^{-1}aua^{-1}$ for all $a\in A_{x,l}.$
The element $g=bu^{-1}\in B$ satisfies $gK=x$ and $aga^{-1}=g$ for all $a\in A_{x,l}$.
\end{proof}

\begin{Lem}
There exists a choice for $g_{1,x}$ that lies in the centralizer of $A_{x,l}$.
\end{Lem}
\begin{proof}
Since $l$ divides $(q-1)$, the Bruhat decomposition implies that $|\underline{G}(\F_q)/\underline{B}(\F_q)|$ equals the order of the Weyl group of $G$ modulo $l$.
Thus, by the proof of \cite{Ve}, Lemma 3.9, the centralizer $\underline{Z}(A_{x,l})$ of the finite group $A_{x,l}$ in $\underline{G}$ is a connected, reductive subgroup.
Let $\underline{S}^{0}\subseteq \underline{G}$ be the connected component of the double centralizer of $A_{x,l}$.
By \textit{loc.cit.~}we know that $\underline{S}^{0}$ is a central subgroup of $\underline{Z}(A_{x,l})$ and that $A_{x,l}$ is a subgroup of $\underline{S}^{0}$ (see \cite{Ve}, Lemma 3.9 (a) for the second claim.)
In particular, the group $\underline{Z}(A_{x,l})$ is also the centralizer of $\underline{S}^{0}$.

As $A_{x,l}$ is a subgroup of $\underline{T}$, its double centralizer, and therefore also $\underline{S}^{0}$, is a subgroup of the double centralizer of $\underline{T}$, which is $\underline{T}$ itself.
Thus, $\underline{S}^{0}$ is a connected subgroup of a a split torus and consequently a split torus itself.
By \cite{BorelTits} Theorem 4.15 (a), the centralizer of a split torus inside a reductive group is a Levi subgroup.
To summarize, we have shown that the centralizer of $A_{x,l}$ in $G$ is the group $M=\underline{M}(F)$ of $F$-valued point of a Levi subgroup $\underline{M}\subseteq\underline{G}$.

The group $A_{x,l}$ is stable under the involution $\sigma$ and therefore its centralizer $M$ is also stable under it.
By Lemma \ref{Step1} we may choose a representative of $x$ that lies in $M$.
Applying Corollary \ref{corollary} to $M$ we see that we may choose $g_{1,x}\in M$, which proves the claim. 
\end{proof}

As a direct consequence of our main theorem we get:
\begin{Cor}\label{exponent}
Let $l$ be a prime such that
\begin{itemize}
\item $q\equiv 1 \bmod l$ and
\item $l$ does not divide the order of the Weyl group of $G$.
\end{itemize}
Then $\mathcal{H}(G)^{\sph}_{\Z/l^r\Z}$ is graded-commutative for all $r$.
\end{Cor}

\begin{Rem}
Let $A_l$ be the $l$-Sylow subgroup of $A$.
Under the additional condition $q \equiv 1 \bmod l^r$ Venkatesh shows that the corestriction map
$$
\Cores^{A_l}_{A_{x,l}}\colon \HH^{\ast}(A_{x,l}, \Z/l^r\Z) \too \HH^{\ast}(A_l, \Z/l^r\Z)
$$
is the zero map whenever the subgroup $A_{x,l}\subseteq A_l$ is a proper subgroup (see \cite{Ve}, Lemma 3.10).
This is a major step in proving the derived Satake isomorphism.
The advantage of our approach is that we can deduce graded-commutativity without assuming the vanishing of the terms corresponding to proper subgroups $A_{y^\prime,l} \subsetneq A_l$ in \eqref{multres}.
\end{Rem}

\section{The case $\GL_2$}
As before, $R$ denotes a commutative ring in which $p$ is invertible.
In Section 2.4 of \cite{Ve} Venkatesh argues that the spherical derived Hecke algebra is `largest' if the equality $q-1= 0$ holds in the coefficient ring $R$.
The following proposition shows graded-commutativity in the opposite case: the hypothesis in the proposition ensures that the derived Hecke algebra is as `small' as possible while still being potentially non-trivially derived.

Combining the proposition - in which we still consider an arbitrary connected, split reductive group - with the results of the previous section we deduce graded-commutativity of the spherical derived Hecke-algebra of the group $\GL_2$ for most torsion coefficients (see Corollary \ref{cor} below.)
 
\begin{Pro}\label{trivial}
Suppose that the order of every proper parabolic subgroup $P\subsetneq\underline{G}(\F_q)$ is invertible in $R$.
Then the spherical derived Hecke algebra $\mathcal{H}(G)^{\sph}_{R}$ is graded-commutative.
\end{Pro}
\begin{proof}
Let $x,y \in G/K$ with $x\neq zy$ for every $z$ in the centre of $G$.
We claim that $\HH^{i}(G_{x,y},R)=0$ for all $i\geq 1$.
As before we may take $x=t\in T$ and $y=1$ and replace $G_{t1}$ by its image in $\underline{G}(\F_q).$
Modulo a $p$-subgroup this is a proper Levi subgroup of $\underline{G}(\F_q)$ and hence, the claim follows by assumption.
In other words, an element $h\in\mathcal{H}(G)^{\sph}_R$ of degree greater than or equal to one is uniquely characterized by its values $h(1,z)$ with $z$ in the centre of $G$.

Let $z$ be in the centre of $G$.
We have $G_{1,z}=K$ and $h(x,y)=h(zx,zy)$ for every $h\in\mathcal{H}(G)^{\sph}_R$ and $x,y \in G/K$.
Let $h_1$ and $h_2$ be elements of $\mathcal{H}(G)^{\sph}_R$ of degree $i$ respectively $j$.
Since the classical spherical Hecke algebra is commutative we may assume that $i$ or $j$ is greater than or equal to one.
Using the explicit formula \eqref{product} for the product we see:
\begin{align*}
(h_1\ast h_2)(1,z)&=(h_1\ast h_2)(z^{-1},1)\\
&=\sum_{t\in K\backslash G/K} \Cores^{K}_{G_{1,t}}(\Res^{G_{1,t}}_{G_{1,t}} (h_1(z^{-1},t))\cup \Res^{G_{1,t}}_{G_{1,t}} (h_2(t,1)))\\
&=(-1)^{ij}\sum_{t\in K\backslash G/K} \Cores^{K}_{G_{1,t}}(h_2(t,1)\cup h_1(z^{-1},t))\\
&=(-1)^{ij}\sum_{t\in K\backslash G/K} \Cores^{K}_{G_{1,t}}(h_2(t,1)\cup h_1(1,zt))\\
&=(-1)^{ij}\sum_{t\in K\backslash G/K} \Cores^{K}_{G_{1,t}}[t^{-1}]^{\ast}(h_2(1,t^{-1})\cup h_1({t}^{-1},z)).
\end{align*}
Again, the summands for $t$ not in the centre vanish and we conclude
\begin{align*}
&\sum_{t\in K\backslash G/K} \Cores^{K}_{G_{1,t}}[t^{-1}]^{\ast}(h_2(1,t^{-1})\cup h_1({t}^{-1},z))\\
=&\sum_{t\in K\backslash G/K} \Cores^{K}_{G_{1,t}}(h_2(1,t^{-1})\cup h_1({t}^{-1},z))\\
=&(h_2\ast h_1)(1,z).
\end{align*}
\end{proof}

\begin{Rem}
Let us assume that $G$ is semisimple. For $x\in G/K$ we put $K_x=G_{x,1}.$
Fixing a set of representatives $[K\backslash G/K]$ for the left $K$-orbits of $G/K$ with $1\in [K\backslash G/K]$ we have a decomposition
\begin{align}\label{decomp}
\mathcal{H}(G)^{\sph}_R=\bigoplus_{x\in [K\backslash G/K]} \HH^{\ast}(K_x,R)
\end{align}
by \cite{Ve}, Section 2.4.
The proof above shows that under the hypothesis of Proposition \ref{trivial}
we have $$\HH^{\ast}(K_x,R)=\HH^{0}(K_x,R)$$
for all $x\neq 1.$
Moreover, multiplication in $\mathcal{H}(G)^{\sph}_R$ is given in terms of \eqref{decomp} as follows:
\begin{itemize}
\item convolution in degree $0$,
\item cup product on $\HH^{\ast}(K,R)$ and
\item $h_1\ast h_2=h_2\ast h_1=0$ if $h_1\in \HH^{0}(K_x,R)$ for $x\neq 1$ and $h_2\in \HH^{i}(K,R)$ for $i>0.$
\end{itemize}
\end{Rem}

\begin{Cor}\label{cor}
Let $n\in \Z$ be an integer with $(n,2p)=1$.
Then the spherical derived Hecke algebras $\mathcal{H}(\GL_2(F))^{\sph}_{\Z/n\Z}$, $\mathcal{H}(\SL_2(F))^{\sph}_{\Z/n\Z}$ and $\mathcal{H}(\PGL_2(F))^{\sph}_{\Z/n\Z}$ are graded-commutative.
\end{Cor}
\begin{proof}
For coprime integers $m,n\geq 1$ we may always decompose
$$\mathcal{H}(G)^{\sph}_{\Z/nm\Z}=\mathcal{H}(G)^{\sph}_{\Z/m\Z}\times \mathcal{H}(G)^{\sph}_{\Z/n\Z}.$$
Thus, it is enough to consider the case $n=l^r$ for a prime $l$ that is coprime to $2p$.

Every parabolic subgroup of $\underline{G}=\GL_2,$ $\SL_2$ or $\PGL_2$ is a Borel subgroup.
We have $|\underline{G}(\F_q)/\underline{B}(\F_q)|=|\mathbb{P}^1(\F_q)|=q+1$ and $|\underline{B}(\F_q)|=q(q-1)^{i}$ with $i=1$ or $i=2$.
Thus, the claim follows either from Corollary \ref{exponent} or Proposition \ref{trivial}.

\end{proof}

\bibliographystyle{abbrv}
\bibliography{bibfile}

\end{document}